\documentclass[10pt,oneside,leqno]{amsart}
\usepackage{amssymb}
\usepackage{amsfonts}
\usepackage{graphicx}
\usepackage{amsmath}
\usepackage{color}

\usepackage{lipsum}
\usepackage{epstopdf}
\usepackage{algorithmic}

\usepackage{appendix}

\usepackage{amsthm}
\usepackage{amsmath}
\usepackage{amssymb}
\usepackage{amsfonts}
\usepackage{amscd}
\usepackage{amsfonts}
\usepackage{amsthm}
\usepackage{graphicx}
\usepackage[T1]{fontenc}
\usepackage[utf8]{inputenc}
\usepackage[french,english]{babel}
\usepackage{color}
\usepackage{marvosym}

\usepackage{amsthm}
\usepackage[export]{adjustbox}

\newtheorem{proposition}{Proposition}
\newtheorem{remark}{Remark}

\newcommand{\func}[1]{\operatorname{#1}}  
\begin{document}

\title[]{Fast and accurate approximations to fractional powers of operators}
\thanks{This work was partially supported by GNCS-INdAM, PRA-University of Pisa
and  FRA-University of Trieste. 
The authors are members of the INdAM research group GNCS}
\keywords{Matrix functions, Gauss-Laguerre rule, Fractional Laplacian}
\author{Lidia Aceto}
\address{Lidia Aceto\\ Universit\`{a} di Pisa\\
 Dipartimento di Matematica, via F. Buonarroti 1/C - 56127 Pisa\\
Italy}
\email{lidia.aceto@unipi.it}
\author{Paolo Novati}
\address{Paolo Novati \\ Universit\`{a} di Trieste\\
Dipartimento di Matematica e Geoscienze, via Valerio 12/1,  34127 Trieste\\
Italy}
\email{novati@units.it}

\begin{abstract}
In this paper we consider some rational approximations to the fractional powers of self-adjoint positive operators, arising from the Gauss-Laguerre rules. We derive practical error estimates that can be used to select a priori the number of Laguerre points necessary to achieve a given accuracy. We also present some numerical experiments to show the effectiveness of our approaches and the reliability of the estimates.
\end{abstract}

\maketitle

\section{Introduction}
The numerical solution of problems involving fractional diffusion can lead to the computation of 
fractional powers of unbounded operators. For instance,  denoting by $\Delta$ the standard Laplace operator and taking $ \alpha \in (0,1),$  the fractional Laplace equation
\begin{equation} \label{lap}
(-\Delta)^\alpha u =f
\end{equation}
on a bounded Lipschitz domain subject to Dirichlet boundary conditions can be solved by computing 
\begin{equation} \label{dec}
\sum_{j=1}^{+\infty} \mu_j^{-\alpha} \langle f, \varphi_j \rangle \varphi_j,
\end{equation}
where $\mu_j$ and $\varphi_j$  are the eigenvalues and the  eigenfunctions of $-\Delta,$ respectively,  and $\langle \cdot, \cdot  \rangle  $  denotes the $L^2$-inner product. In practice, in this situation the fractional derivative can be identified by the fractional power. Keeping in mind this kind of applications, in this work we are interested in the numerical approximation of ${\mathcal{L}}^{-\alpha}, \alpha \in (0,1).$ Here ${\mathcal{L}}$ is a self-adjoint positive operator acting in an Hilbert space ${\mathcal H}$ in which  the eigenfunctions of ${\mathcal{L}}$ form an orthonormal basis of ${\mathcal{H}},$ so that  ${\mathcal{L}}^{-\alpha}$ can be written through the spectral decomposition of ${\mathcal{L}}$ as in (\ref{dec}). 

In recent years, this problem has been studied  by many authors. Due to the properties of the function $\lambda^{-\alpha}, \lambda \in [\ell, +\infty),\ell>0,$ the most effective approaches are those based on a rational approximation of this function. In the continuous setting of unbounded operators, methods based on the best uniform rational approximation (BURA) of functions closely related to $\lambda^{-\alpha}$ have been considered, for example, in \cite{HLMMV, HLMMP,  HML, HM} by using a modified version of the Remez algorithm. Another class of methods relies on quadrature rules for the integral representation of $\lambda^{-\alpha}$ \cite{AN0,AN,ABDN,Bo,V19,V18}. 
Very recently, time stepping methods for a parabolic reformulation of the fractional diffusion equation (\ref{lap}) given in \cite{V15}  have  also been interpreted in \cite{H} as a rational approximation of $\lambda^{-\alpha}.$ 

In this paper, starting from the integral representation given in \cite[Eq. (4)]{Bo}
\begin{equation} \label{intOR}
\mathcal{L}^{-\alpha }=\frac{2\sin (\alpha \pi )}{\pi }\int_{0}^{+\infty
}t^{2\alpha -1}({\mathcal I}+t^{2}\mathcal{L})^{-1}dt ,\qquad \alpha \in
(0,1),
\end{equation}
where ${\mathcal I}$ is the identity operator in ${\mathcal H},$  after suitable changes of variables we consider an alternative rational approximation based on the truncated Gauss-Laguerre rule. In order to construct the truncated approach, we exploit the error analysis of the standard Gauss-Laguerre rule based on the theory of analytic functions originally introduced in \cite{BA}.  We are able to show that in the operator norm 
the error decay like
\[
\exp(-c m^{1/2})
\]
where $m$ is the number of inversions and $c = 3.6 \alpha^{1/2} $ (cf. (\ref{bal})). In this view, the formula seems to be competitive with the Sinc quadrature studied in \cite{Bo} in which $c= \pi (1-\alpha)^{1/2}\alpha^{1/2}$ by  Remark 3.1 of the same paper. However, it appears to be slightly slower than that based on the analysis given in   \cite{S} and related to the BURA approach in which $c = 2 \pi  (1-\alpha)^{1/2}$ although the approach presented here does not suffer from the instability of Remez algorithm. 

We also present a further modification of the truncated Gauss-Laguerre rule, called equalized rule, that allows to further reduce the number of inversions to achieve the same accuracy, especially when $\alpha\le 1/2.$ 
\\

The paper is structured as follows. In Section \ref{Sec 2}  we present the Gauss-Laguerre approach. In Sections \ref{Sec 3}-\ref{Sec 4}, starting from the error analysis based on the theory of analytic functions, we present the error estimate attainable with the Gauss-Laguerre approach for the approximation of $\lambda^{-\alpha}.$ The analysis is then extended in Section \ref{Sec 5}  to the case of the operator  ${\mathcal L}^{-\alpha}.$ Finally,  the truncated rules are proposed in Section \ref{Sec 6}.


\section{The Gauss-Laguerre approach} \label{Sec 2}
As already said in the introduction, we start from the integral representation given in (\ref{intOR}).
Setting $y=\ln t $  we obtain 
\begin{equation}
\mathcal{L}^{-\alpha }=\frac{2\sin (\alpha \pi )}{\pi }\int_{-\infty
}^{+\infty }e^{{2\alpha }y}({\mathcal I}+e^{2y}\mathcal{L})^{-1}dy,\qquad \alpha \in
(0,1).  \label{mat0}
\end{equation}%
Now we consider separately the two integrals%
\begin{equation*}
\int_{-\infty }^{0}e^{{2\alpha }y}({\mathcal I}+e^{2y}\mathcal{L})^{-1}dy,\quad
\int_{0}^{+\infty }e^{{2\alpha }y}({\mathcal I}+e^{2y}\mathcal{L})^{-1}dy
\end{equation*}%
and consider the changes of variable ${2\alpha }y=-x$ and $2(1-\alpha )y=x$
respectively, to obtain 
\begin{eqnarray*}
\int_{-\infty }^{0}e^{{2\alpha }y}({\mathcal I}+e^{2y}\mathcal{L})^{-1}dy &=&\frac{1}{%
2\alpha }\int_{0}^{+\infty }e^{-x}({\mathcal I}+e^{-x/\alpha }\mathcal{L})^{-1}dx,   \\
\int_{0}^{+\infty }e^{{2\alpha }y}({\mathcal I}+e^{2y}\mathcal{L})^{-1}dy &=&\frac{1}{%
2(1-\alpha) }\int_{0}^{+\infty }e^{-x}(e^{-x/(1-\alpha )}{\mathcal I}+\mathcal{L}%
)^{-1}dx.
\end{eqnarray*}%
Consequently, setting
\begin{eqnarray}
I^{(1)}(\lambda ) &:=&\int_{0}^{+\infty
}e^{-x}(1+e^{-x/\alpha }\lambda )^{-1}dx,  \label{int1} \\
I^{(2)}(\lambda ) &:=&\int_{0}^{+\infty
}e^{-x}(e^{-x/(1-\alpha )}+\lambda )^{-1}dx,  \label{int2}
\end{eqnarray}%
the operator in (\ref{mat0}) can be written as 
\begin{equation}
\mathcal{L}^{-\alpha }=\frac{\sin (\alpha \pi )}{\alpha \pi } I^{(1)}(%
\mathcal{L})+\frac{\sin (\alpha \pi )}{(1-\alpha) \pi }I^{(2)}(\mathcal{L}).  \label{mat}
\end{equation}
It is easy to check that $I^{(1)}(\mathcal{L}) \rightarrow  \mathcal{I}$ as $\alpha \rightarrow 0$ and $I^{(2)}(\mathcal{L}) \rightarrow  \mathcal{L}^{-1 }$ as $\alpha \rightarrow 1.$ 

By applying the $n$-point Gauss-Laguerre rule  to both integrals with respect to
the weight function $\omega (x)=e^{-x},$  with weights $w_{j}^{(n)}$ and
nodes $\vartheta _{j}^{(n)}$ (in ascending order), we obtain the following $(2n-1,2n)$ rational
approximation 
\begin{equation}
\mathcal{L}^{-\alpha }\approx\frac{\sin (\alpha \pi )}{\alpha \pi } R_{n-1,n}^{(1)}(\mathcal{L})+ \frac{\sin (\alpha \pi )}{(1-\alpha) \pi }R_{n-1,n}^{(2)}(\mathcal{L})
=:R_{2n-1,2n}(\mathcal{L}),  \label{rapp}
\end{equation}
where
\begin{eqnarray*}
R_{n-1,n}^{(1)}(\lambda ) &=&\sum_{j=1}^{n}w_{j}^{(n)}%
\left( 1+e^{-{\vartheta _{j}^{(n)}}/{\alpha }}\lambda \right) ^{-1}, \\
R_{n-1,n}^{(2)}(\lambda ) &=&\sum_{j=1}^{n}w_{j}^{(n)}%
\left( e^{-{\vartheta _{j}^{(n)}}/{(1-\alpha) }}+\lambda \right) ^{-1}.
\end{eqnarray*}%
Clearly, formula (\ref{rapp}) implies that using $n$ points we have to perform $2n$ inversions.


\section{Error analysis for a general function}  \label{Sec 3}

In order to obtain an estimate of the error for the rational approximation defined in  (\ref%
{rapp}), we consider the approach introduced in \cite{BA} and based on the
theory of analytic functions. Assuming to work with a general function $f$
and then to consider the $n$-point Gauss-Laguerre rule $I_{n}(f)$ for 
\[ 
I(f)=\int_{0}^{+\infty }e^{-x}f(x)dx,
\]
we define the remainder as $E_{n}(f)=I(f)-I_{n}(f)$. For any given $R>1$, the
equation%
\begin{equation*}
\func{Re}(\sqrt{-z})=\ln R
\end{equation*}%
represents a parabola in the complex plane, that we denote by $\Gamma _{R}$,
symmetric with respect to the real axis, with vertex in $-\left( \ln
R\right) ^{2}$ and convexity oriented towards the positive real axis. By
writing $z=a+ib$, the above equation reads%
\begin{equation*}
a=\left( b^{2}-4\left( \ln R\right) ^{4}\right) \frac{1}{4\left( \ln
R\right) ^{2}}.
\end{equation*}%
The parabola degenerates to $[0,+\infty )$ as $R\rightarrow 1$. The theory given
in \cite{BA} states that, if for a given $R$ the function $f$ is analytic on
or within $\Gamma _{R}$ except for a pair of simple poles, $z_{0}$ and its
conjugate $\overline{z_{0}}$, then%
\begin{equation}
E_{n}(f)\approx -4\pi \func{Re}\left\{ r e^{-z_0}\left[ \exp \left(
\sqrt{-z_0}\right) \right] ^{-2 \sqrt{\bar n}}\right\} ,   \label{bf}
\end{equation}%
where $r$ is the residue of $f(z)$ at $z_{0}$ and 
\begin{equation} \label{numia}
\bar n=4n+2.
\end{equation}

This result follows from the fact that $E_{n}(f)$ can be written as a
contour integral%
\begin{equation*}
E_{n}(f)=\frac{1}{2\pi i}\int_{\Gamma }\frac{q_{n}(z)}{L_{n}(z)}f(z)dz,
\end{equation*}%
where $L_{n}(z)$ is the Laguerre polynomial, $q_{n}(z)$ is the so-called
associated function defined by%
\begin{equation*}
q_{n}(z)=\int\nolimits_{0}^{+\infty }\frac{e^{-x}L_{n}(x)}{z-x}dx,\quad
z\notin \lbrack 0,+\infty ),
\end{equation*}%
and $\Gamma $ is a contour containing $[0,+\infty )$ with the additional
property that no singularity of $f(z)$ lies on or within this contour (see \cite[\S 4.6]{DR} for a background).

Denoting by $C_{1}$ and $C_{2}$ two arbitrary small circles surrounding the
two poles the idea is then to define $\Gamma =\Gamma _{R}\cup C_{1}\cup
C_{2} $. In order to run this contour in the counterclockwise direction, one
can artificially add three line segments as shown in Figure \ref{Figure1} to
connect the circles with the parabola. Then, following the black and the red
arrows, the integrals along the line segments cancel and we obtain%
\begin{equation}  \label{rem}
E_{n}(f)=\frac{1}{2\pi i}
\left\{\int_{\Gamma _{R}}-\int_{C_{1}}-\int_{C_{2}} \right\}\frac{q_{n}(z)}{L_{n}(z)}f(z)dz.
\end{equation}

\begin{figure}[htbp]
  \centering
\includegraphics[width=0.80\textwidth]{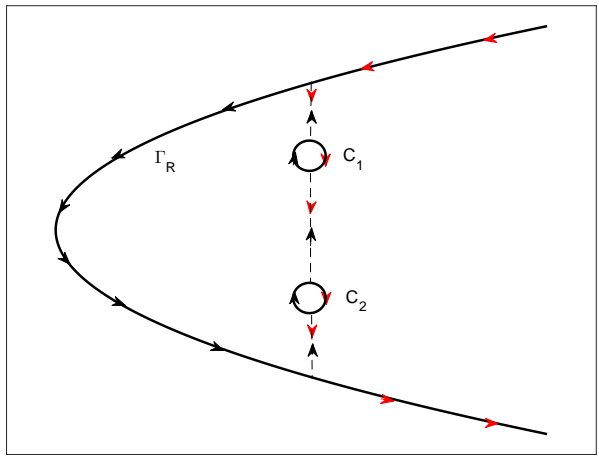}
 \caption{Contour chosen for a function $f $ analytic on or within the parabola $\Gamma _{R}$ with the exception of two simple and conjugated poles located inside $C_{1}$ and $C_{2},$ respectively.}
\label{Figure1}
\end{figure}
At this point, the estimate is based on the relation given in \cite[Eq. (5.4)]{E}, namely
\begin{equation*}
\frac{q_{n}(z)}{L_{n}(z)}=2\pi e^{-z}\left[ \exp \left( \sqrt{-z}\right) \right] ^{-2\sqrt{\bar n}}\left( 1+O\left( \frac{1}{n}%
\right) \right) ,\quad  z\notin \lbrack 0,+\infty ),
\end{equation*}
Since 
\begin{equation} \label{erre}
\left[ \exp \left(  \func{Re}\left( \sqrt{-z}\right)\right) \right] ^{-2\sqrt{\bar n}}=R^{-2\sqrt{\bar n}},\quad \text{\ for }z\in \Gamma _{R}, 
\end{equation}
the contribution on the parabola is given by
\begin{eqnarray} \label{fn}
&& \frac{1}{2\pi i}\int_{\Gamma _{R}}\frac{q_{n}(z)}{L_{n}(z)}%
f(z)dz =  R^{-2\sqrt{\bar n}}   \\
&\times& \frac{1}{i}  \int_{\Gamma _{R}}
e^{-z}  
\left[ \exp \left(  i \func{Im}\left( \sqrt{-z}\right)\right) \right] ^{-2\sqrt{\bar n}}
f(z) dz \left( 1+O\left( \frac{1}{n}%
\right) \right) :=\phi(n).   \nonumber  
\end{eqnarray}
In addition, using the residue theorem we have
\begin{eqnarray*}
&&\frac{1}{2\pi i} \left\{\int_{C_{1}} + \int_{C_{2}} \right\}  e^{-z}\left[ \exp \left( \sqrt{-z}\right) \right] ^{-2\sqrt{\bar n}} f(z)dz  \\
&=&  \func{Res}\left( e^{-z}\left[ \exp \left( \sqrt{-z}\right) \right] ^{-2\sqrt{\bar n}}  f(z),z_{0}\right)  \\
&+&  \func{Res}\left( e^{-z}\left[ \exp \left( \sqrt{-z}\right) \right] ^{-2\sqrt{\bar n}}  f(z),\overline{z_{0}}\right) \\
&=&2 \func{Re}\left( \func{Res}\left( e^{-z}\left[ \exp \left( \sqrt{-z}\right) \right] ^{-2\sqrt{\bar n}}  f(z),z_{0}\right) \right) \\
&=&2 \func{Re}\left( \func{Res}\left( f(z),z_{0}\right) 
e^{-z_0}\left[ \exp \left( \sqrt{-z_0}\right) \right] ^{-2\sqrt{\bar n}} 
\right).
\end{eqnarray*}
Therefore from (\ref{rem}), by taking into account (\ref{fn}), we obtain
\begin{eqnarray*}
E_{n}(f)&=& -4\pi \func{Re}\left( \func{Res}\left( f(z),z_{0}\right) 
e^{-z_0}\left[ \exp \left( \sqrt{-z_0}\right) \right] ^{-2\sqrt{\bar n}} 
\right)  \left( 1+O\left( \frac{1}{n} \right) \right)\\ &+& \phi(n).
\end{eqnarray*}
Obviously, this implies the formula (\ref{bf})  whenever the contribution from the parabola $\Gamma_R$ (i.e., $\phi(n)$) can be considered negligible. As for the modulus of the error, observing that 
\begin{eqnarray*}
&& \left| \func{Re}\left( \func{Res}\left( f(z),z_{0}\right) 
e^{-z_0}\left[ \exp \left( \sqrt{-z_0}\right) \right] ^{-2\sqrt{\bar n}} 
\right)  \right|  \leq  \\
&\leq&  \left|  \func{Res}\left( f(z),z_{0}\right) 
e^{-z_0}  \right|   \left[ \exp  \left(  \func{Re}\left( \sqrt{-z_0}\right) \right)\right] ^{-2\sqrt{\bar n}},
\end{eqnarray*}
we have
\begin{eqnarray*}
| E_{n}(f) |&\leq& 4\pi \left|  \func{Res}\left( f(z),z_{0}\right) 
e^{-z_0}  \right|   \left[ \exp  \left(  \func{Re}\left( \sqrt{-z_0}\right) \right)\right] ^{-2\sqrt{\bar n}} 
\left( 1+O\left( \frac{1}{n} \right) \right) \\ &+& |\phi(n)|.
\end{eqnarray*}
Since hereafter we assume that 
\[ 
\int_{\Gamma _{R}} |e^{-z}   f(z)| dz 
\] 
is bounded,  from  (\ref{fn}) we obtain (see  (\ref{numia}) and (\ref{erre}))
\[
\frac{\left| \phi(n) \right| }{ \left[ \exp  \left(  \func{Re}\left( \sqrt{-z_0}\right) \right)\right] ^{-2\sqrt{\bar n}} } \le  \frac{cR^{-2\sqrt{\bar n}}}{\left[ \exp  \left(  \func{Re}\left( \sqrt{-z_0}\right) \right)\right] ^{-2\sqrt{\bar n}}}  =O\left( \exp(-n^{1/2}) \right)
\]
and then 
\begin{equation} \label{mainerror}
| E_{n}(f) |  \leq 4\pi \left|  \func{Res}\left( f(z),z_{0}\right) 
e^{-z_0}  \right|   \left[ \exp  \left(  \func{Re}\left( \sqrt{-z_0}\right) \right)\right] ^{-2\sqrt{\bar n}} 
\left( 1+O\left( \frac{1}{n} \right) \right).
\end{equation}


\section{Error analysis for $\lambda ^{-\alpha }$}  \label{Sec 4}

From (\ref{mat}) and (\ref{rapp}) and defining
\begin{equation} \label{errori}
\varepsilon _{n}^{(i)}\left( \lambda \right) =\left\vert I^{(i)}(\lambda
)-R_{n-1,n}^{(i)}(\lambda )\right\vert ,\quad i=1,2,
\end{equation}%
we can write
\begin{equation}
\left\vert \lambda ^{-\alpha }-R_{2n-1,2n}(\lambda )\right\vert \leq \frac{%
\sin (\alpha \pi )}{\alpha \pi } \varepsilon _{n}^{(1)}\left( \lambda \right)
+ \frac{\sin (\alpha \pi )}{(1-\alpha) \pi }\varepsilon _{n}^{(2)}\left( \lambda \right)  .  \label{line}
\end{equation}
Hence, using the results of the previous section we can develop the error analysis by working separately on
the two integrals $I^{(i)}(\lambda )$, $i=1,2.$ 
 

\subsection{First integral $I^{(1)}(\protect\lambda )$}

The function involved in (\ref{int1}) is 
\begin{equation} \label{f1}
f(z)=(1+e^{-z/\alpha }\lambda )^{-1},
\end{equation}%
whose poles are given by%
\begin{equation*}
z_{k}=\alpha \ln \lambda +i(2k+1)\alpha \pi ,\quad k\in \mathbb{Z}.
\end{equation*}%
They are equally spaced along the line $\func{Re}(z)=\alpha \ln \lambda $,
symmetric with respect to the real axis, and the closest to the real axis
are $z_{0}=\alpha \ln \lambda +i\alpha \pi $ and $z_{-1}=\overline{z_{0}}$.
It is immediate to verify that there exists $R>1$ such that the corresponding
parabola $\func{Re}(\left( -z\right) ^{1/2})=\ln R$ contains only the
poles $z_{0}$ and $\overline{z_{0}}$ in its interior and that such an $R$
satisfies%
\[
\frac{\alpha }{2}\left( \sqrt{\left( \ln \lambda \right) ^{2}+\pi ^{2}}-\ln
\lambda \right) <\left( \ln R\right) ^{2}<\frac{\alpha }{2}\left( \sqrt{%
\left( \ln \lambda \right) ^{2}+9\pi ^{2}}-\ln \lambda \right). \label{rmax}
\]
These bounds follow by imposing $z_{0}\in \Gamma _{R}$ (the left one) and $%
z_{1}=\alpha \ln \lambda +i 3\alpha \pi \in \Gamma _{R}$ (the right one).

In order to apply (\ref{mainerror}), first we observe that 
\begin{eqnarray*}
\left( -z_{0}\right) ^{1/2} &=&\left[- \left( \alpha \ln \lambda +i\alpha \pi
\right)\right]^{1/2} \\
&=&\sqrt{\frac{\alpha }{2}}\left( \gamma ^{-}\left( \lambda \right) -i\gamma
^{+}\left( \lambda \right) \right) ,  
\end{eqnarray*}%
where%
\begin{equation}
\gamma ^{\pm }\left( \lambda \right) =\sqrt{\sqrt{\left( \ln \lambda \right)
^{2}+\pi ^{2}}\pm \ln \lambda }.  \label{gamma}
\end{equation}%
Then, recalling that $z_{0}/\alpha= \ln \lambda +i \pi ,$ we write 
\[
1+e^{-z/\alpha} \lambda = 1 - e^{-(z-z_0)/\alpha} = \frac{z-z_0}{\alpha} \sum_{j=0}^{+\infty} \frac{(-1)^j (z-z_0)^j}{\alpha^j (j+1)!}.
\]
In this case,  the residue of the function given in (\ref{f1}) at the simple pole $z_0$ is given by
\[
\func{Res}\left( f(z),z_{0}\right) =\lim_{z\rightarrow z_{0}}\frac{z-z_{0}}{1+e^{-z/\alpha
}\lambda }=\alpha.
\]
Therefore, from (\ref{mainerror}) we have
\begin{equation}
\varepsilon _{n}^{(1)}\left( \lambda \right) \le 4\pi \alpha \lambda ^{-\alpha
}\exp \left( -\gamma ^{-}\left( \lambda \right) \left( 2\alpha \bar n
\right) ^{1/2}\right) \left( 1+O\left(  \frac{1}{n}\right) \right).  \label{en1}
\end{equation}


\subsection{Second integral $I^{(2)}(\protect\lambda )$}

The function to consider in this case is 
\[ \label{f2}
f(z)=(e^{-z/(1-\alpha )}+\lambda )^{-1},
\]
whose poles are given by%
\begin{equation*}
z_{k}=-(1-\alpha )\ln \lambda +i(2k+1)(1-\alpha )\pi ,\quad k\in \mathbb{Z}.
\end{equation*}%
The only difference with respect to the integral  $I^{(1)}(\protect\lambda )$ is that the poles
have now a negative real part. Anyway, as before we can easily find a
parabola containing in its interior only the poles $z_0=-(1-\alpha )\ln \lambda
+ i(1-\alpha )\pi $ and its conjugate. We have now%
\begin{equation*}
\left( -z_{0}\right) ^{1/2}=\sqrt{\frac{1-\alpha }{2}}\left( \gamma
^{+}\left( \lambda \right) +i\gamma ^{-}\left( \lambda \right) \right),  
\end{equation*}%
where $\gamma ^{\pm }\left( \lambda \right) $ are defined in (\ref{gamma}).
As for the residue at $z_{0}$ we easily find that $\func{Res}\left( f(z),z_{0}\right) =(1-\alpha)/\lambda $. Using again (\ref{mainerror}) we have%
\begin{equation}
\varepsilon _{n}^{(2)}\left( \lambda \right) \leq 4\pi (1-\alpha)\lambda ^{-\alpha
}\exp \left( -\gamma ^{+}\left( \lambda \right) \left( 2(1-\alpha )\overline{%
n}\right) ^{1/2}\right) \left( 1+O\left(  \frac{1}{n}\right) \right).   \label{en2}
\end{equation}%

Finally,  plugging  in (\ref{line}) the bounds (\ref{en1}) and (\ref{en2}) we have the following result.
\begin{proposition}
Let   $\gamma ^{\pm }\left( \lambda \right) $  be defined in (\ref{gamma}) and $\bar n=4n+2.$  
Denoting by 
\begin{eqnarray}
g_{n}^{(1)}(\lambda ) &:=&\lambda ^{-\alpha }\exp \left( -\gamma ^{-}\left(
\lambda \right) \left( 2\alpha \bar n \right) ^{1/2}\right),  \label{g1}
\\
g_{n}^{(2)}(\lambda ) &:=&\lambda ^{-\alpha }\exp \left( -\gamma ^{+}\left(
\lambda \right) \left( 2(1-\alpha )\bar n \right) ^{1/2}\right)
\label{g2}
\end{eqnarray}%
the $\lambda $-dependent factors of $\varepsilon _{n}^{(1)}\left(
\lambda \right) $ and $\varepsilon _{n}^{(2)}\left( \lambda \right)$,
respectively, then we have
\begin{equation}
 \left\vert  \lambda ^{-\alpha } - R_{2n-1,2n}(\lambda )\right\vert \leq 
4  \sin(\alpha \pi ) \left[ g_{n}^{(1)}(\lambda )+g_{n}^{(2)}(\lambda )\right] 
\left( 1+O\left(  \frac{1}{n}\right) \right).   \label{esti}
\end{equation}
\end{proposition}

In order to verify the estimate provided in (\ref{esti}), in Figure \ref{Figure2} we consider an example with $\lambda =10.$
Here and below, nodes and weights of the Gauss-Laguerre rule have been computed using the Matlab function \texttt{GaussLaguerre.m} given in \cite{G}.
\begin{figure}[htbp]
  \centering
\includegraphics[width=0.80\textwidth]{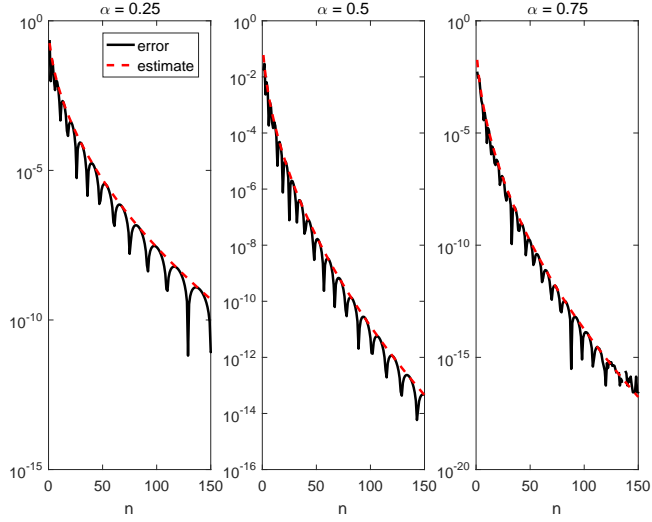}
 \caption{Absolute  error and its estimate given by (\ref{esti}) for $\lambda =10.$}
\label{Figure2}
\end{figure}
 

\section{Error analysis for $\mathcal{L}^{-\protect\alpha }$}  \label{Sec 5}

For simplicity, from now on we assume that $\sigma (\mathcal{L})\subseteq \lbrack 1,+\infty ).$
Since $\mathcal{L}$ is self-adjoint and positive, regarding  the error we have 
\begin{equation}
\left\Vert \mathcal{L}^{-\alpha }-R_{2n-1,2n}(\mathcal{L})\right\Vert \leq
\max_{\lambda \geq 1}\left\vert \lambda ^{-\alpha }-R_{2n-1,2n}(\lambda
)\right\vert,  \label{ltos}
\end{equation}
where $\left\Vert  \cdot \right\Vert$ denotes the operator norm in $\mathcal H.$
By  (\ref{esti}) we must therefore study  the functions $g_{n}^{(i)}(\lambda )$, 
$i=1,2,$ for $\lambda \ge 1.$ In particular, this means to study 
the functions $\gamma ^{\pm }\left( \lambda \right) $ (see (\ref{g1}) and (\ref{g2})). 
By (\ref{gamma}),  it is immediate to see that $\gamma ^{-}\left( \lambda \right) \rightarrow 0$ and $\gamma
^{+}\left( \lambda \right) \rightarrow +\infty $ as $\lambda \rightarrow
+\infty $. As consequence,
the function $g_{n}^{(1)}(\lambda )$ has exactly one maximum at a certain $%
\lambda _{n}>1$, whereas $g_{n}^{(2)}(\lambda )$ is monotone decreasing,  independently of $\alpha $ and $n.$
At this point, in order to compute the right hand side in (\ref{ltos}) the first step consists in finding the
point of maximum $\lambda _{n}.$ 
\begin{proposition} \label{prolam}
Let $\lambda _{n}$ be the maximum of the function $g_{n}^{(1)}(\lambda ).$  Then, for $n$ large enough
\[ 
\lambda _{n}=\widetilde{\lambda }_{n}\left( 1+O\left( n^{-1/3}\right) \right), 
\]
where%
\[
\widetilde{\lambda }_{n}= 
\exp \left( \left( \left( \frac{\overline{n}\pi ^{2}}{4\alpha}\right) ^{2/3}-\pi ^{2}\right) ^{1/2}\right). \label{ls}
\]
\end{proposition}

\begin{proof}
By imposing $\displaystyle{\frac{d}{d\lambda}} g_{n}^{(1)}(\lambda )=0$, after some manipulation we arrive at the equation%
\begin{equation} \label{miaeq}
\frac{\sqrt{\left( \ln \lambda \right) ^{2}+\pi ^{2}}-\ln \lambda }{\left(
\ln \lambda \right) ^{2}+\pi ^{2}}=\frac{2\alpha }{\bar n},
\end{equation}
whose solution is denoted by $\lambda_n.$ 
Since 
\begin{eqnarray}
\frac{\sqrt{\left( \ln \lambda \right) ^{2}+\pi ^{2}}-\ln \lambda }{\left(
\ln \lambda \right) ^{2}+\pi ^{2}} &=&\frac{\pi ^{2}}{\left( \left( \ln
\lambda \right) ^{2}+\pi ^{2}\right) \left( \sqrt{\left( \ln \lambda \right)
^{2}+\pi ^{2}}+\ln \lambda \right) }  \label{ide} \\
&\geq &\frac{\pi ^{2}}{2\left( \left( \ln \lambda \right) ^{2}+\pi
^{2}\right) ^{3/2}}, \notag
\end{eqnarray}%
by (\ref{miaeq}) we first observe that there exists a constant $c$ independent
of $n$ such that $\left( \ln \lambda _{n}\right) ^{3}\geq cn$, for $n$ large
enough. Writing%
\begin{equation*}
\ln \lambda =s\sqrt{\left( \ln \lambda \right) ^{2}+\pi ^{2}},
\end{equation*}%
where%
\begin{equation}
s=s(\lambda )=\frac{1}{\sqrt{1+\left( \frac{\pi }{\ln \lambda }\right) ^{2}}},  \label{esse}
\end{equation}%
by (\ref{miaeq}) and (\ref{ide}) we obtain%
\begin{equation*}
\frac{\pi ^{2}}{\left( \left( \ln \lambda \right) ^{2}+\pi ^{2}\right)
^{3/2}\left( 1+s\right) }=\frac{2\alpha }{\overline{n}}.
\end{equation*}%
As consequence%
\begin{equation*}
\lambda _{n}=\exp \left( \left( \left( \frac{\overline{n}\pi ^{2}}{2\alpha
(1+s(\lambda _{n}))}\right) ^{2/3}-\pi ^{2}\right) ^{1/2}\right).
\end{equation*}%
Since asymptotically $\left( \ln \lambda _{n}\right) ^{2}\geq cn^{2/3}$, from
(\ref{esse}) we have%
\begin{equation*}
s(\lambda _{n})=1+O(n^{-2/3})
\end{equation*}%
and therefore%
\begin{equation}
\lambda _{n}=\exp \left( \left( \left( \frac{\overline{n}\pi ^{2}}{4\alpha }%
\right) ^{2/3}-\pi ^{2}+O(1)\right) ^{1/2}\right).   \label{eln}
\end{equation}%
Writing%
\begin{equation*}
\left( \left( \frac{\overline{n}\pi ^{2}}{4\alpha }\right) ^{2/3}-\pi
^{2}+O(1)\right) ^{1/2}=\left( \left( \frac{\overline{n}\pi ^{2}}{4\alpha }%
\right) ^{2/3}-\pi ^{2}\right) ^{1/2}+\sigma_{n}
\end{equation*}
we easily find that%
\begin{equation*}
\sigma_{n}=O\left( n^{-1/3}\right) .
\end{equation*}%
Finally, we obtain the result since%
\begin{equation*}
\lambda _{n}=\exp \left( \left( \left( \frac{\overline{n}\pi ^{2}}{4\alpha )}%
\right) ^{2/3}-\pi ^{2}\right) ^{1/2}\right) \exp (\sigma_{n}).
\end{equation*}
\end{proof}

This approximation is rather good as it can be observed in Figure \ref{Figure3}
where we plot $\ln \lambda _{n}$ and $\ln \widetilde{\lambda }_{n}$ for $%
n=10,11,\dots,120.$ Here the value of  $\lambda _{n}$ which verifies (\ref{miaeq}) has been numerically computed by  using a
nonlinear solver. 
\begin{figure}[htbp]
  \centering
\includegraphics[width=0.80\textwidth]{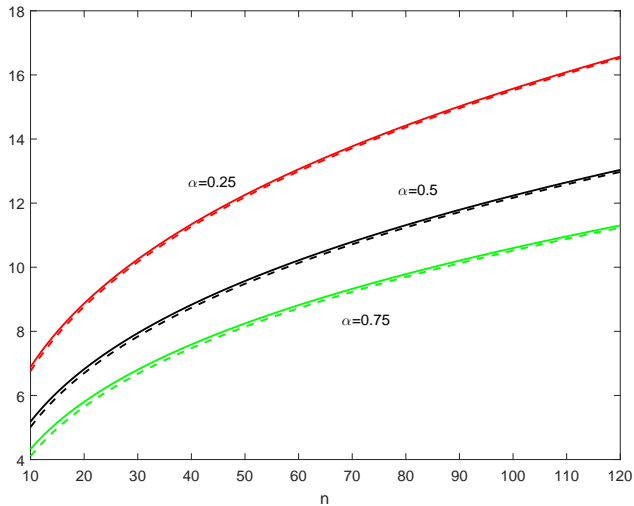}
 \caption{Comparison between $\ln \lambda _{n}$ (solid lines) and $\ln {\tilde \lambda_n }$ (dahed lines) for $n=10,11, \dots,120$.}
\label{Figure3}
\end{figure}

\begin{proposition} \label{pro3} Let $g_{n}^{(1)}(\lambda )$ and $g_{n}^{(2)}(\lambda )$ be the functions defined in (\ref%
{g1})  and (\ref{g2}), respectively. Then, 
\begin{eqnarray}
\max_{\lambda \geq 1}g_{n}^{(1)}(\lambda ) &=&  g_{n}^{(1)}(\lambda _{n}) 
=\exp \left( -3\left( n\alpha ^{2}\pi ^{2}\right)
^{1/3}\right) \left( 1+O\left( n^{-1/3}\right) \right), \label{dec1} \\
\max_{\lambda \geq 1}g_{n}^{(2)}(\lambda ) &=&g_{n}^{(2)}(1)  
 = \exp \left( -\left( 8\pi (1-\alpha )n\right) ^{1/2}\right) \left(
1+O\left( n^{-1/2}\right) \right) .  \label{dec3}
\end{eqnarray}
\end{proposition}

\begin{proof}
First of all we need to evaluate $\gamma ^{-}\left( \lambda _{n}\right)
\left( 2\alpha \overline{n}\right) ^{1/2}$. Using (\ref{gamma}) and  (\ref{miaeq}) we have%
\begin{equation*}
\left( \gamma ^{-}\left( \lambda _{n}\right) \right) ^{2}=\frac{2\alpha }{%
\overline{n}}\left( \left( \ln \lambda _{n}\right) ^{2}+\pi ^{2}\right) .
\end{equation*}%
By (\ref{eln}) we also have%
\begin{equation}
\left( \ln \lambda _{n}\right) ^{2}+\pi ^{2}=\left( \frac{\overline{n}\pi
^{2}}{4\alpha }\right) ^{2/3}+O(1)  \label{ln}
\end{equation}%
and hence%
\begin{eqnarray*}
\gamma ^{-}\left( \lambda _{n}\right)  &=&\left( \left( \frac{\alpha }{2%
\overline{n}}\right) ^{1/3}\pi ^{4/3}+O(n^{-1})\right) ^{1/2} \\
&=&\left( \frac{\alpha }{2\overline{n}}\right) ^{1/6}\pi ^{2/3}\left(
1+O(n^{-2/3})\right).
\end{eqnarray*}%
Consequently,
\begin{equation*}
\gamma ^{-}\left( \lambda _{n}\right) \left( 2\alpha \overline{n}\right)
^{1/2}=(2 \overline{n}\alpha ^2 \pi^2)^{1/3}\left(
1+O(n^{-2/3})\right).
\end{equation*}%
Using the result obtained in Proposition \ref{prolam}, we can write
\begin{eqnarray*}
\lambda _{n}^{-\alpha } 
&=&\exp \left( -\alpha \left( \left( \left( \frac{\overline{n}\pi ^{2}}{%
4\alpha }\right) ^{2/3}-\pi ^{2}\right) ^{1/2}\right) \right) \left(
1+O\left( n^{-1/3}\right) \right)  \\
&=&\exp \left( -\left( \frac{\overline{n}\alpha ^{2}\pi ^{2}}{4}\right)
^{1/3}\right) \left( 1+O\left( n^{-1/3}\right) \right). 
\end{eqnarray*}%
Therefore, we have%
\begin{eqnarray*}
g_{n}^{(1)}(\lambda _{n}) &=&\exp \left( -\left( \frac{\overline{n}\alpha
^{2}\pi ^{2}}{4}\right) ^{1/3}\right) \left( 1+O\left( n^{-1/3}\right)
\right)   \notag \\
&\times&\exp \left( -( 2 \overline{n} \alpha ^2 \pi^2 )^{1/3}\left(
1+O(n^{-2/3})\right) \right)   \notag \\
&=&\exp \left( -\left( \overline{n}\alpha ^{2}\pi ^{2}\right) ^{1/3}\left(
4^{-1/3} +2^{1/3}\right) \right) \left( 1+O\left( n^{-1/3}\right) \right).
\label{dec0}
\end{eqnarray*}%
Finally, recalling that $\overline{n}=4n+2$ we obtain the result.

As for the function $g_{n}^{(2)}(\lambda )$, the situation is much simpler.
Indeed, since it is monotone decreasing using (\ref{gamma}) and (\ref{g2}) we have that%
\begin{eqnarray*}
\max_{\lambda \geq 1}g_{n}^{(2)}(\lambda ) &=&g_{n}^{(2)}(1)  \notag =\exp \left( -\left( 2\pi (1-\alpha )\overline{n}\right) ^{1/2}\right) 
\label{dec2} \\
&=&\exp \left( -\left( 8\pi (1-\alpha )n\right) ^{1/2}\right) \left(
1+O\left( n^{-1/2}\right) \right).  
\end{eqnarray*}
\end{proof}

Finally, we can prove the following result.
\begin{proposition} \label{prop4}
Let $R_{2n-1,2n}(\mathcal{L})$ be the rational approximation given in (\ref{rapp}). Then, with respect to the operator norm in ${\mathcal H}$  we have for $n$ large enough 
\begin{equation}
\left\Vert \mathcal{L}^{-\alpha }-R_{2n-1,2n}(\mathcal{L})\right\Vert  \leq 4\sin
(\alpha \pi )\exp \left( -3\left( n\alpha ^{2}\pi ^{2}\right) ^{1/3}\right)
\left( 1+O\left( n^{-1/3}\right) \right).   \label{e5}
\end{equation}
\end{proposition}

\begin{proof}
First of all, by comparing (\ref{dec1}) with (\ref{dec3}) for $n
$ large enough we can write 
\begin{equation*}
\frac{g_{n}^{(2)}(1)}{g_{n}^{(1)}(\lambda _{n})}\leq \frac{1}{n}. 
\end{equation*}%
Therefore, 
\begin{eqnarray*}
\max_{\lambda \geq 1}\left( g_{n}^{(1)}(\lambda )+g_{n}^{(2)}(\lambda
)\right)  &\leq &\max_{\lambda \geq 1}g_{n}^{(1)}(\lambda )+\max_{\lambda
\geq 1}g_{n}^{(2)}(\lambda ) \\
&\leq &g_{n}^{(1)}(\lambda _{n})+g_{n}^{(2)}(1) \\
&=&g_{n}^{(1)}(\lambda _{n})\left( 1+O\left( n^{-1}\right) \right).
\end{eqnarray*}
By  Propostion \ref{pro3} we find the result.
\end{proof}

To test the estimate just given in  Proposition \ref{prop4}  we work with the operator 
\begin{equation} \label{optest}
\mathcal{L}=\left[ \func{diag}(1,2,\dots,100)\right] ^{8}
\end{equation} 
so that $\sigma (\mathcal{L})\subseteq %
\left[ 1,10^{16}\right] $. In Figure \ref{Figure4} we plot the error and its estimate (\ref{e5}) with
respect to the number of inversions, that is, $2n.$ From now on, for discrete operators the error is plotted with respect to  the Euclidean matrix norm.
\begin{figure}[htbp]
  \centering
\includegraphics[width=0.80\textwidth]{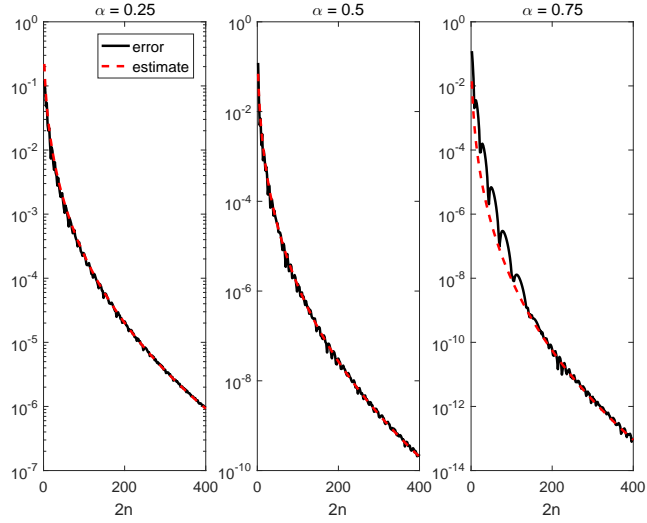}
 \caption{Error and its estimate given by (\ref{e5})  for  the operator defined in (\ref{optest}). }
\label{Figure4}
\end{figure}

Notwithstanding the above result, experimentally (see Figure \ref{Figure5}) it
is immediate to observe that
\begin{eqnarray*}
\max_{\lambda \geq 1}\left( g_{n}^{(1)}(\lambda )+g_{n}^{(2)}(\lambda ) \right)&\approx
&\max \left( g_{n}^{(1)}(\lambda _{n}),g_{n}^{(2)}(1)\right).
\end{eqnarray*}%

This because the contribution of a function in correspondence of the maximum
of the other one is negligible. In order to understand whenever $%
g_{n}^{(2)}(1)$ may be greater than $g_{n}^{(1)}(\lambda _{n})$ for some
values of $n$ and $\alpha$ (as in Figure \ref{Figure5} for $\alpha =0.75$) we just need to
compare $g_{n}^{(2)}(1)$ with $g_{n}^{(1)}(1)$.
\begin{figure}[htbp]
  \centering
\includegraphics[width=0.80\textwidth]{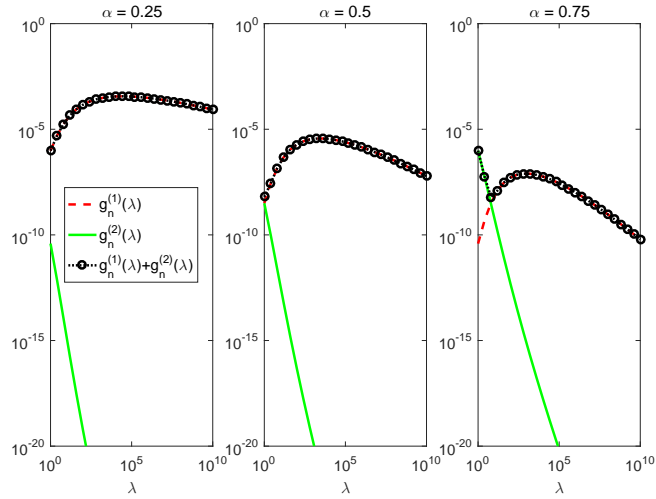}     
 \caption{Behavior of the functions $g_{n}^{(1)}(\lambda),g_{n}^{(2)}(\lambda),g_{n}^{(1)}(\lambda)+g_{n}^{(2)}(\lambda)$ for $n=30.$}
\label{Figure5}
\end{figure}
Using (\ref{gamma}), (\ref{g1}) and (\ref{g2}) the equation $%
g_{n}^{(1)}(1)=g_{n}^{(2)}(1)$ is approximatively equivalent to
\begin{equation*}
\exp \left( -\left( 2\pi \alpha \overline{n}\right) ^{1/2}\right) =\exp
\left( -\left( 2\pi (1-\alpha )\overline{n}\right) ^{1/2}\right) ,
\end{equation*}%
whose solution is $\alpha =1/2$ independently of $n.$ This means that for $\alpha \leq 1/2$%
\begin{equation*}
\max_{\lambda \geq 1} \left(g_n^{(1)}(\lambda )+g_n^{(2)}(\lambda ) \right)\approx
g_n^{(1)}(\lambda _{n})
\end{equation*}%
and therefore the error decays like $\exp \left( -cn^{1/3}\right) $ for some absolute constant  $c$ (cf. (%
\ref{dec1})), whereas for $\alpha >1/2$ the situation is a bit more
complicate. By comparing (\ref{dec1}) with (\ref{dec2}) we have that
asymptotically $g_n^{(2)}(1)$ decay faster than $g_n^{(1)}(\lambda _{n})$, so,
after a certain $n^{\ast }$ the decay rate is still of type $\exp \left(
-cn^{1/3}\right) $ also for $\alpha >1/2.$ Anyway, for $n\leq n^{\ast }$ the decay rate is of type $\exp \left(
-cn^{1/2}\right).$
The integer $n^{\ast }$ comes from the solution with respect to $n$ of%
\begin{equation*}
g_{n}^{(1)}(\lambda _{n})=g_{n}^{(2)}(1).
\end{equation*}
Using Proposition~\ref{pro3} we can estimate it by solving%
\begin{equation*}
\exp \left( -3\left( n\alpha ^{2}\pi ^{2}\right) ^{1/3}\right)
=\exp \left( -\left( 8\pi (1-\alpha )n\right) ^{1/2}\right).
\end{equation*}%
We easily find%
\begin{equation}
n^{\ast }\approx 4.5\frac{\alpha ^{4}}{(1-\alpha )^{3}}.  \label{nss}
\end{equation}
The previous considerations can be summarized as follows:
\begin{equation} \label{stimanum}
\left\Vert \mathcal{L}^{-\alpha }-R_{2n-1,2n}(\mathcal{L})\right\Vert
\approx 4\sin (\alpha \pi )S(n,\alpha), 
\end{equation}
where 
\begin{eqnarray}
S(n,\alpha) &=&\left\{ 
\begin{array}{ll}
g_{n}^{(1)}(\lambda_n), & (\forall n  \wedge \alpha \leq 1/2) \vee (n>n^{\ast}  \wedge \alpha >1/2) \\ 
g_{n}^{(2)}(1), & (n\leq n^\ast \wedge \alpha >1/2)
\end{array}%
\right. \label{errb}
\end{eqnarray}
(see (\ref{dec1}) and (\ref{dec3})).


\section{Truncated approaches}  \label{Sec 6}

The idea of truncating the Gauss-Laguerre rule is clearly not new and is
essentially consequence of the fact that the weights decay exponentially.
Among the existing papers on this point we recall \cite{Ber}, where a
truncated approach has been used for the computation of the Laplace
transform, and \cite{MM}, where the authors develop the error analysis of
the truncated Gauss-Laguerre rule for a general $f$ absolutely continuous.

Here we focus on the case where $f$ is an arbitrary continuous function that satisfies $0\leq f(x)\leq 1,$
since this is the case of the functions that appear in the definition of $I^{(i)}(\lambda), i=1,2.$
In fact, we clearly have that  for $\lambda \geq 1$ (see (\ref{int1}) and (\ref{int2}))
\[
0 \leq (1+e^{-x/\alpha }\lambda )^{-1}\leq 1, \qquad \quad
0 \leq (e^{-x/(1-\alpha )}+\lambda )^{-1}\leq 1.
\]
Suppose that a sequence of error approximations $\left\{ \varepsilon
_{n}\right\} _{n\geq 1}$ is available, that is,%
\begin{equation} \label{if01}
\left\vert I(f)-I_{n}(f)\right\vert \leq \varepsilon _{n},
\end{equation}%
where now $I_{n}(f)$ is the $n$-point Gauss-Laguerre approximation of $I(f),$ with $0\leq f(x)\leq 1.$
Since 
\[
\int_0^{+\infty }e^{-x} f(x) dx \le \int_0^{+\infty }e^{-x} dx,
\]
let $s_{n}$ be the solution of 
\[\label{intex}
\int_{s_{n}}^{+\infty }e^{-x}dx=\varepsilon _{n},
\]
that is,%
\begin{equation}
s_{n}=-\ln \varepsilon _{n}.  \label{sn}
\end{equation}
We consider the truncated rule
\begin{eqnarray*}
I_{k_{n}}(f) &=&\sum_{j=1}^{k_{n}}w_{j}^{(n)}f(\vartheta _{j}^{(n)}) \\
&=& I_{n}(f)-\sum_{j=k_{n}+1}^{n}w_{j}^{(n)}f(\vartheta
_{j}^{(n)}),
\end{eqnarray*}%
where $k_{n}\leq n$ is the smallest integer such that $\vartheta _{j}^{(n)}\geq
s_{n}$ for $j\geq k_{n}$.
Therefore,
\begin{eqnarray*}
\left\vert I(f)-I_{k_{n}}(f)\right\vert &=&\left\vert
I(f)-I_{n}(f)+\sum_{j=k_{n}+1}^{n}w_{j}^{(n)}f(\vartheta
_{j}^{(n)})\right\vert \\
&\leq &\left\vert I(f)-I_{n}(f)\right\vert +\sum_{j=k_{n}+1}^{n}w_{j}^{(n)}.
\end{eqnarray*}%
Using the bound \cite[Eqs. (2.4) and (2.7)]{MO}
\begin{equation*}
w_{j}^{(n)}\leq C(\vartheta _{j}^{(n)}-\vartheta _{j-1}^{(n)})e^{-\vartheta
_{j}^{(n)}},\quad j=2,\dots, n
\end{equation*}%
where $C$ is a constant independent of $n$, we have (see (\ref{sn}))
\begin{equation*}
\sum_{j=k_{n}+1}^{n}w_{j}^{(n)}\leq Ce^{-\vartheta _{k_{n}}^{(n)}}\leq
Ce^{-s_{n}}=C\varepsilon _{n},
\end{equation*}%
so that finally%
\begin{equation*}
\left\vert I(f)-I_{k_{n}}(f)\right\vert \leq (1+C)\varepsilon _{n}.
\end{equation*}

\begin{remark}
Experimentally one can easily check that the approximation%
\begin{equation*}
w_{j}^{(n)}\approx (\vartheta _{j}^{(n)}-\vartheta
_{j-1}^{(n)})e^{-\vartheta _{j}^{(n)}}
\end{equation*}%
is very accurate and hence in the numerical experiments we take $C=1.$
\end{remark}


\subsection{A balanced approach}

Let $k_{n}\leq n$ be the smallest integer such that (see (\ref{e5}))
\[
\vartheta _{j}^{(n)}\geq -\ln \left( 4\sin (\alpha \pi) \exp \left( -3\left( n\alpha ^{2}\pi ^{2}\right) ^{1/3}\right)  \right)  \left( 1+O\left( n^{-1/3}\right) \right),\quad j \ge k_{n}.
\]
Using the above theory we have that for $n$ large enough 
\begin{eqnarray} \label{errbal}
\left\Vert \mathcal{L}^{-\alpha }-R_{2k_{n}-1,2k_{n}}(\mathcal{L}%
)\right\Vert &\le& 4 (1+C) \sin
(\alpha \pi )\exp \left( -3\left( n\alpha ^{2}\pi ^{2}\right) ^{1/3}\right)  \\
&\times& \left( 1+O\left( n^{-1/3}\right) \right). \nonumber
\end{eqnarray}%
In order to derive error estimates with respect to $k_{n}$, that is, with
respect to the number of inversions, we first need to prove the following result.
\begin{proposition}
For $k$ large enough, the $k$-th root of the Laguerre polynomial of degree $n$  satisfies
\begin{equation}
\vartheta _{k}^{(n)}=c_{k}\frac{k^{2}\pi ^{2}}{4n}(1+O(n^{-2})),\quad
1<c_{k}\leq \left( 1+\frac{1}{k}\right) ^{2}. \label{pro1}  
\end{equation}
\end{proposition}

\begin{proof}
First of all we need to study the asymptotic behavior of the 
roots of $J_0(z),$ the Bessel function of the first kind of order 0.  By \cite[Eq. (1.71.7)]{Sz}%
\begin{equation*}
J_0(z)=\left( \frac{2}{\pi z}\right) ^{1/2}\cos \left( z-\frac{\pi }{4}\right)
+O(z^{-3/2}),
\end{equation*}%
we observe that there is a root, say $j_{k},$ in $I=[\pi/2+k\pi ,(k+1)\pi ]$ since $J_0(z)$ changes sign. Now, let%
\begin{equation} \label{zk}
z_{k}=\frac{3}{4}\pi +k\pi \in I
\end{equation}%
be the solution of $\cos \left( z-\frac{\pi }{4}\right) =0$, so that $%
J_0(z_{k})=O(k^{-3/2})$. Therefore,
\begin{equation*} 
z_{k}-j_{k}=\frac{J_0(z_{k})}{J_0^{\prime }(\xi )},\quad \xi \in I.
\end{equation*}%
Now, since
\begin{equation*}
J_0^{\prime }(\xi )=\left( \frac{2}{\pi }\right) ^{1/2}\left[ -\frac{1}{2}\xi
^{-3/2}\cos \left( \xi -\frac{\pi }{4}\right) -\sin \left( \xi -\frac{\pi }{4%
}\right) \xi ^{-1/2}\right] +O(\xi ^{-5/2}).
\end{equation*}%
and
\begin{equation*}
\left\vert \sin \left( \xi -\frac{\pi }{4}\right) \right\vert \geq \frac{%
\sqrt{2}}{2}
\end{equation*}%
we deduce that $J_0^{\prime }(\xi )=O(k^{-1/2}).$ From the above considerations we have
\begin{equation*}
z_{k}-j_{k}=O(k^{-1})
\end{equation*}
and then, using (\ref{zk}) we get 
\[
j_{k}^{2} =\left( \frac{3}{4}\pi +k\pi +O(k^{-1})\right) ^{2} =\left( k\pi \right) ^{2}\left( 1+\frac{3}{4k}+O(k^{-2})\right) ^{2}.
\]
By \cite[Eq. (22.16.8)]{AS}
\begin{equation*}
\vartheta _{k}^{(n)}=\frac{j_{k}^{2}}{4n+2}\left[ 1+\frac{j_{k}^{2}}{4\left(
4n+2\right) ^{2}}\right] +O\left( n^{-5}\right) 
\end{equation*}
we obtain the result.
\end{proof}

Now we want to solve with respect to  $k$ 
\begin{equation}
\vartheta _{k}^{(n)}=-\ln \left( 4\sin (\alpha \pi) \exp \left( -3\left( n\alpha ^{2}\pi ^{2}\right) ^{1/3}\right)  \right)  \left( 1+O\left( n^{-1/3}\right) \right).  \label{equk}
\end{equation}
For $k$ large enough, by (\ref{pro1}), the solution of (\ref{equk})
satisfies%
\begin{equation}
c_{k}\frac{k^{2}\pi ^{2}}{4n}(1+O(n^{-2}))=-\ln \left( 4\sin (\alpha \pi
)\right) +3\left( n\alpha ^{2}\pi ^{2}\right) ^{1/3}+O(n^{-1/3}),
\label{equk2}
\end{equation}%
that is%
\begin{equation*}
c_{k}\frac{k^{2}\pi ^{2}}{4n}=3\left( n\alpha ^{2}\pi ^{2}\right)
^{1/3}(1+O(n^{-1/3})).
\end{equation*}%
By the definition of $c_{k}$ we thus have $k\sim n^{2/3}$ and therefore%
\begin{equation*}
k^{2}\left( 1+O(k^{-1/2})\right) =12\alpha ^{2/3}\pi ^{-4/3}n^{4/3}
\end{equation*}%
that leads to%
\begin{equation*}
n^{1/3}=\frac{k^{1/2}}{12^{1/4}\alpha ^{1/6}\pi ^{-1/3}}\left(
1+O(k^{-1/2})\right).
\end{equation*}%
Using this value in (\ref{errbal}) we find%
\begin{eqnarray*}
 \left\Vert \mathcal{L}^{-\alpha }-R_{2k-1,2k}(\mathcal{L})\right\Vert  &\leq&
4(1+C)\sin (\alpha \pi ) \\
&\times& \exp \left( -3\frac{\pi }{12^{1/4}}\alpha
^{1/2}k^{1/2}\left( 1+O(k^{-1/2})\right) \right) \left( 1+O(k^{-1/2})\right) \\
&\leq& 4(1+C){\hat C}\sin (\alpha \pi )\exp \left( -3\frac{\pi }{12^{1/4}}\alpha
^{1/2}k^{1/2}\right),
\end{eqnarray*}%
where the constant ${\hat C}$ takes into account of the term $\left(
1+O(k^{-1/2})\right) $.

We remark however that the above analysis can be simplified by neglecting
the terms $\ln \left( 4\sin (\alpha \pi )\right) $ and $c_{k}$ in (\ref%
{equk2}), and solving directly%
\begin{equation*}
\frac{k^{2}\pi ^{2}}{4n}=3\left( n\alpha ^{2}\pi ^{2}\right) ^{1/3}.
\end{equation*}%
Using the floor function, we denote by 
\begin{equation}   \label{kn1}
k_{n}^{(1)} = \left \lfloor  2\sqrt{3}  \left( \frac{ \alpha n^2}{\pi^2} \right)^{1/3} \right \rfloor,
\end{equation}
that experimentally is confirmed to be a value rather closed to $k_{n}$, in
a reasonable range of values of $\alpha $, say $\alpha \in [0.05,0.95]$, leading to a method that is almost 
 indistinguishable from the one with $k_{n}$. Since%
\begin{equation}
n \approx \frac{\pi }{\alpha ^{1/2}} \left( \frac{k_{n}^{(1)}}{2\sqrt{3}}\right)^{3/2}  \label{knn}
\end{equation}%
using (\ref{errbal}) we find%
\begin{equation}  \label{bal}
\left\Vert \mathcal{L}^{-\alpha }-R_{2k_{n}^{(1)}-1,2k_{n}^{(1)}}(\mathcal{L}%
)\right\Vert  \approx 4(1+C) \sin (\alpha \pi )\exp \left( - 3.6 \alpha^{1/2}  \left( 2 k_{n}^{(1)} \right)^{1/2} \right). 
\end{equation}

By using again the  operator (\ref{optest}), in Figure \ref{Figure6} we compare the two errors 
provided by applying the $n$-point Gauss-Laguerre rule and the corresponding balanced formula, that is
\[
\left\Vert \mathcal{L}^{-\alpha }-R_{{2j-1},2j}(\mathcal{L})\right\Vert, \qquad j=n, k_n^{(1)}.
\] 
We can observe the great improvement in terms of computational cost attainable with the truncated approach.
\begin{figure}[htbp]
  \centering
\includegraphics[width=0.80\textwidth]{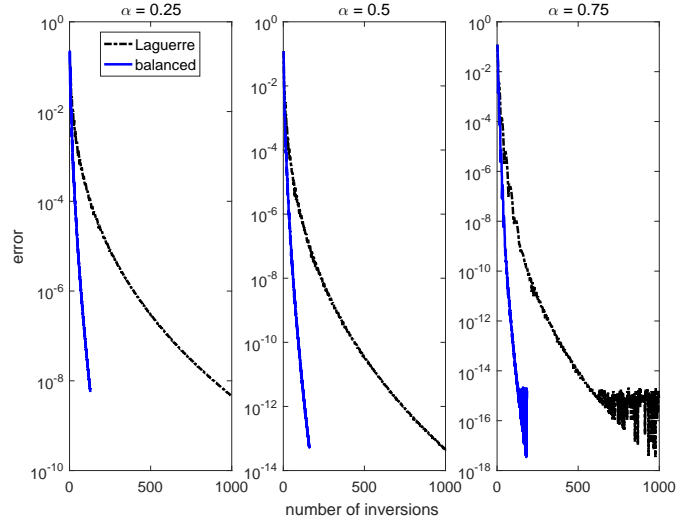}
 \caption{$\left\Vert \mathcal{L}^{-\alpha }-R_{{2j-1},2j}(\mathcal{L})\right\Vert$ vs  the number of inversions $2j,$ for  $j=n$ (Laguerre) and $j=k_{n}^{(1)}$ (balanced).} \label{Figure6}
\end{figure}
In Figure \ref{Figure7}  we focus the attention on the truncated (balanced) approach. We plot the error and its estimate (\ref{bal}) with $C=1$  with respect to the number of inversions, that is, $2k_n^{(1)}.$ The results show the accuracy of the estimate. 
\begin{figure}[htbp]
  \centering
\includegraphics[width=0.80\textwidth]{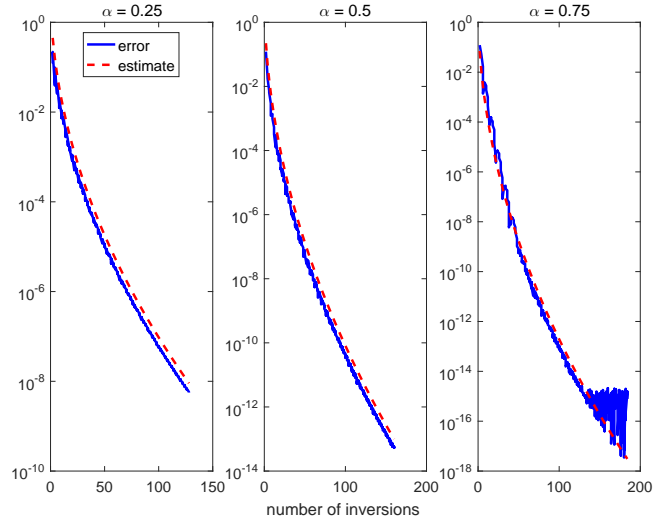}
 \caption{Error and its estimate given by (\ref{bal}) for the  operator defined in (\ref{optest}).}
\label{Figure7}
\end{figure}

When $\alpha >1/2$ the above estimate may be optimistic for $n\leq n^{\ast
}$ (cf. (\ref{nss})). Working with (\ref{stimanum})-(\ref{errb}) with $S(n,\alpha)=g_{n}^{(2)}(1)$ and following the same analysis that starts from (\ref{equk}), by  (\ref{dec3}) we find that 
$k = 2  (1-\alpha )^{1/4} \left(2 n /\pi\right)^{3/4} $
and then the value
\begin{equation}
k_{n}^{(2)}:= 2 \left \lfloor  (1-\alpha )^{1/4}   \left( \frac{2n}{\pi}\right)^{3/4}   \right \rfloor \label{kn2}
\end{equation}%
is very close to $k_{n}.$ Therefore, 
\begin{eqnarray}  \label{bal2}
&&\left\Vert \mathcal{L}^{-\alpha }-R_{2 k_{n}^{(2)} -1,2 k_{n}^{(2)}}(\mathcal{L}%
)\right\Vert  \approx 4(1+C)\sin (\alpha \pi ) \\
&\times& \exp \left( -2.96(1-\alpha
)^{1/3} \left( 2k_n^{(2)} \right)^{2/3}\right),  \qquad \mbox{ for } \, (n \le n^{\ast })  \wedge (\alpha >1/2), \nonumber
\end{eqnarray}%
which expresses an initial convergence very fast with respect to the number of inversions. 
For $\alpha >1/2,$ one should use the first $k_n^{(2)}$  
Laguerre points for $n\le n^{\ast }$ and then switch to the first $k_n^{(1)}$  
for $n> n^{\ast }.$ 
Anyway, experimentally it can be observe
that the corresponding method does not offer a valuable improvement with
respect to the choice of the first $k_n^{(1)},$ independently of $\alpha $ and $n$. 

Therefore, the balanced approach that we propose is the one based on (\ref{kn1}), and reported in the figures, with error estimate given by (\ref{bal})  independently of $\alpha $ and $n$. 


\subsection{An equalized approach}

The idea is to work separately on the two integrals and hence to consider
approximations of the type%
\begin{equation*}
\mathcal{L}^{-\alpha }\approx \frac{\sin (\alpha \pi )}{\alpha \pi } 
R_{k_{n_{1}}-1,k_{n_{1}}}^{(1)}(\mathcal{L})+  \frac{\sin (\alpha \pi )}{(1-\alpha) \pi }  R_{k_{n_{2}}-1,k_{n_{2}}}^{(2)}(%
\mathcal{L}),
\end{equation*}%
in which $R_{k_{n_{i}}-1,k_{n_{i}}}^{(i)}(\lambda )$, $i=1,2$, represents
the truncated Gauss-Laguerre rule for $I^{(i)}(\lambda )$ based on the first $%
k_{n_{i}}$ roots of the Laguerre polynomials of degree $n_i.$ 
For $n_{1}\neq n_{2}$ we use then different sets of points, and clearly the
total number of inversions is now $k_{n_{1}}+k_{n_{2}}$. 

We first consider the case where, for a
given $n$, $\varepsilon_{n}^{(1)}({\lambda})/\alpha\geq \varepsilon_{n}^{(2)}(\lambda)/(1-\alpha)$ (cf. (\ref{errori}) and (\ref{line})) and
we define $n_{1}=n$. Then, we evaluate $k_{n_{1}}=k_{n_{1}}^{(1)}$ as in (\ref{kn1})
and we approximate $I^{(1)}(\mathcal{L})$ with $R_{k_{n_{1}}-1,k_{n_{1}}}^{(1)}(%
\mathcal{L})$. Then, we find $n_{2}$ ($\leq n_{1}$) such that 
\[
g_{n_{1}}^{(1)}({\lambda_{n_{1}}}) = g_{n_2}^{(2)}(1)
\]
that is,
\begin{equation}
\exp \left(-3\left( n_{1}\alpha ^{2}\pi ^{2}\right) ^{1/3}\right) =  \exp \left( -\left( 8\pi (1-\alpha )n_{2}\right) ^{1/2}\right),  \label{ee}
\end{equation}%
(cf.   (\ref{dec1}) and (\ref{dec3})). At this point we compute as in (\ref{kn2})%
\begin{equation}
k_{n_{2}}=k_{n_2}^{(2)}=  2 \left \lfloor  (1-\alpha )^{1/4}   \left( \frac{2n_2}{\pi}\right)^{3/4} \right \rfloor,  \label{knb}  
\end{equation}%
and use the Gauss-Laguerre rule $R_{k_{n_{2}}-1,k_{n_{2}}}^{(2)}(\mathcal{L})$ for
the second integral. Clearly, for each $n$ the error  estimate for the equalized  approach remains the one of the
balanced approach given by (\ref{bal}), but now we have less inversions. 
In this view, we have to find the relationship between $k_{n_{1}}$ and $k_{n_{2}}.$
From (\ref{ee}) we get
\begin{equation*}
n_{2}=\frac{9}{8}\pi ^{1/3}\frac{\alpha ^{4/3}}{1-\alpha }n_{1}^{2/3}
\end{equation*}
so that using   (\ref{knb}) we can express $k_{n_{2}}$ in terms of $n_1.$ Then, 
by (\ref{knn})  we obtain
\[
k_{n_{2}} \approx \frac{3.09}{\alpha ^{3/4}(1-\alpha )^{1/2}}k_{n_{1}}^{3/4}.
\]
from which we deduce that   $(k_{n_{1}}+k_{n_{2}}) \leq 2k_{n_{1}}.$

As for the case $\varepsilon_{n}^{(1)}({\lambda})/\alpha < \varepsilon_{n}^{(2)}(\lambda)/(1-\alpha)$ the arguments follow the same line. Let $n_{2}=n
$ and compute the second integral with $R_{k_{n_{2}}-1,k_{n_{2}}}^{(2)}(%
\mathcal{L})$. Then,  solving (\ref{ee}) with respect to $n_1$  ($\leq n_{2}$)
we obtain
\begin{equation*}
n_{1}=\frac{\left( 8(1-\alpha )\right) ^{3/2}}{27\alpha ^{2}\pi ^{1/2}}%
n_{2}^{3/2}.
\end{equation*}
Consequently, as in (\ref{kn1})%
\[
k_{n_{1}}=k_{n_1}^{(1)}=  \left \lfloor  2\sqrt{3}  \left( \frac{ \alpha n_1^2}{\pi^2} \right)^{1/3} \right \rfloor,
\]
and we compute the first integral with $R_{k_{n_{1}}-1,k_{n_{1}}}^{(1)}(%
\mathcal{L})$. Using (\ref{kn2}) we also have%
\[
k_{n_{2}}=2 \left \lfloor  (1-\alpha )^{1/4}   \left( \frac{2n_2}{\pi}\right)^{3/4}   \right \rfloor
\]
and therefore, collecting the above expressions we finally obtain%
\[
k_{n_{1}}  \approx 0.61  \frac{(1-\alpha )^{2/3}}{\alpha} k_{n_{2}}^{4/3}.
\]
As before, the error estimate for the equalized approach is the one of the balanced approach given by (\ref%
{bal2}) but the number of inversions that we have to consider  is now $(k_{n_{1}}+k_{n_{2}})\leq 2k_{n_{2}}$.

In Figure \ref{Figure8} we consider the comparison between  our two truncated approaches together with Sinc rule analyzed in \cite{Bo}.
\begin{figure}[htbp]
  \centering
\includegraphics[width=0.80\textwidth]{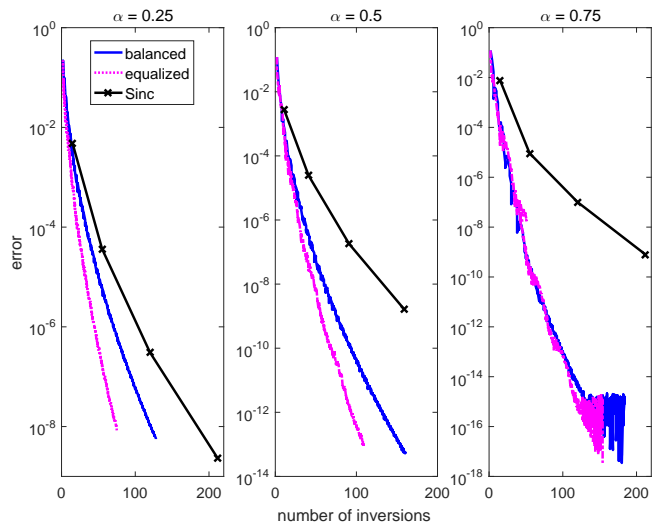}
 \caption{Comparison between the errors provided by the balanced and equalized approaches with the Sinc quadrature studied in \cite{Bo}.}
\label{Figure8}
\end{figure}

\section{Conclusions}

In this work we have considered the construction of very fast methods based on the Gauss-Laguerre rule and we have been able to provide accurate error estimates that can be used to a priori select the number of points to use. We observe that while all the experiments concern the artificial example (\ref{optest}), other tests on finite difference discretizations of the Laplace operator have essentially led to  identical results.


\end{document}